\documentclass{amsart}
\usepackage{amscd,amsmath,amssymb}
\usepackage{pstricks}
\usepackage{latexsym,amsbsy,mathrsfs}
\usepackage{xy}
\usepackage{xypic}
\usepackage{pgf,tikz}

\newtheorem{thm}{Theorem}[section]
\newtheorem{lem}[thm]{Lemma}

\newtheorem{prop}[thm]{Proposition}

\theoremstyle{definition}

\newtheorem{rem}[thm]{Remark}

\numberwithin{equation}{thm}

\begin{document}
\title[ idempotent completion of $n$-ANGULATED CATEGORIES  ]
{ idempotent completion of $n$-ANGULATED CATEGORIES}

\author{Zengqiang Lin}
\address{ School of Mathematical sciences, Huaqiao University,
Quanzhou\quad 362021,  China.} \email{zqlin@hqu.edu.cn}

\thanks{This work was supported  by the Science Foundation of Huaqiao University (Grants No. 2014KJTD14)}

\subjclass[2010]{18E30}

\keywords{  $n$-angulated category; idempotent completion; mapping cone.}

\begin{abstract}
 We show that the idempotent completion of an $n$-angulated category  admits a unique $n$-angulated structure such that the inclusion is an $n$-angulated functor, which satisfies a universal property.
\end{abstract}

\maketitle

\section{Introduction}
Let $n$ be an integer greater than or equal to three. In 2013, Geiss, Keller and Oppermann  introduced the notion of $n$-angulated categories to axiomatize the properties of some $(n-2)$-cluster tilting subcategories of  triangulated categories. By definition, an $n$-angulated category is an additive category $\mathcal{C}$ equipped with an automorphism $\Sigma$ of $\mathcal{C}$ and a class $\Theta$ of $n$-$\Sigma$-sequences which satisfies four axioms, denoted (N1)-(N4) (see \cite[Definition 2.1]{[GKO]}). When $n=3$, an $n$-angulated category is nothing but a triangulated category, where the four axioms are denoted by (TR1)-(TR4) respectively. On \cite[Theorem 1]{[GKO]}, there is a standard construction of $n$-angulated categories. Other examples of $n$-angulated categories can be found in \cite{[BJT],[L]}.

The first aim and motivation of this note is to get more examples of $n$-angulated categories. We want to construct new $n$-angulated categories from known ones. It has been proved by Balmer and Schlichting that the idempotent completion of a triangulated category admits a natural triangulated structure \cite[Theorem 1.5]{[BS]}. The second motivation of this note is to extend \cite[Theorem 1.5]{[BS]} from 3 to $n$. We will show that the idempotent completion of an $n$-angulated category admits a unique $n$-angulated structure such that the inclusion is an $n$-angulated functor, which satisfies a universal property; see Theorem \ref{thm}.

Now we recall some facts used in the proof of \cite[Theorem 1.5]{[BS]}. For (TR1)(c), given an endomorphism $(p,q,r)$ of a triangle $X_\bullet$, if $p^2=p$ and $q^2=q$, then we can use the method of lifting idempotents to find an endomorphism $(p,q,s)$ of $X_\bullet$ such that $s^2=s$. For (TR4), the proof depends on a basic fact that each morphism fits into a triangle uniquely up to isomorphism. But for $n$-angulated categories, the above two facts are not true in general.  Therefore, the original proof fails in our setting.
 For proving (N1)(c) and (N4), we develop some useful facts on $n$-angles in pre-$n$-angulated categories.

This note is organized as follows. In Section 2, we collect some facts on $n$-angles in pre-$n$-angulated categories. In Section 3, we recall some facts on idempotent completion of an additive category, then state and prove our main theorem.

\section{Some facts on $n$-angles}

In this section, we assume that $(\mathcal{C},\Sigma,\Theta)$ is a pre-$n$-angulated category, that is, $\Theta$ satisfies (N1)-(N3). We recall that the elements of $\Theta$ are called $n$-$angles$ and $\Theta$ must contain all contractible $n$-$\Sigma$-sequences \cite[Lemma 2.2(a)]{[BJT]}. For convenience, we denote by $X_\bullet[1]$ the left rotation of the $n$-$\Sigma$-sequence $X_\bullet$. Assume that $\varphi_\bullet:X_\bullet\rightarrow Y_\bullet$ is a morphism of $n$-$\Sigma$-sequences, we denote by $C(\varphi_\bullet)$ the mapping cone of $\varphi_\bullet$. We can find other unmentioned terminology and notations in \cite{[GKO]}.

\begin{lem}\label{1.1}
Let $\Theta$ be a class of $n$-$\Sigma$-sequences satisfying (N1)(a),(N1)(b),(N2) and (N3).
Assume that $$X_\bullet=(X_1\xrightarrow{\left(
                         \begin{smallmatrix}
                           f_1 \\
                           g_1 \\
                         \end{smallmatrix}
                       \right)}X_2\oplus Y_2\xrightarrow{(f_2, g_2)} X_3\xrightarrow{f_3}\cdots\xrightarrow{f_{n-1}} X_n\xrightarrow{f_n}\Sigma X_1)\in\Theta,$$
where $Y_2\in\mathcal{C}$.

(1) If $g_1=0$ or $g_2$ is a section, then
$X_\bullet\cong X'_\bullet\oplus Y'_\bullet$, where
$$X_\bullet'=(X_1\xrightarrow{f_1}X_2\xrightarrow{f_{21}}X_3'\xrightarrow{f_{31}}X_4\xrightarrow{f_{4}}\cdots\xrightarrow{f_{n-1}} X_n\xrightarrow{f_n}\Sigma X_1)\in\Theta,$$
$$Y'_\bullet=(0\rightarrow Y_2\xrightarrow{1} Y_2\rightarrow 0\rightarrow\cdots\rightarrow 0\rightarrow0).$$

(2) If $g_2=0$ or $g_1$ is a retraction, then
$X_\bullet\cong X''_\bullet\oplus Y_\bullet''$, where $$X_\bullet''=(X'_1\xrightarrow{f_{11}}X_2\xrightarrow{f_2}X_3\xrightarrow{f_3}\cdots\xrightarrow{f_{n-1}} X_n\xrightarrow{f_{n1}}\Sigma X'_1)\in\Theta,$$
$$Y''_\bullet=(Y_2\xrightarrow{1} Y_2\rightarrow 0\rightarrow \cdots\rightarrow 0\rightarrow\Sigma Y_2).$$
\end{lem}

\begin{proof}
We only prove (1) since (2) can be proved similarly.
We first show that $g_1=0$ implies that $g_2$ is a section. In fact, if $g_1=0$, then the following diagram
$$\xymatrix{
X_1\ar[r]^{\left(
                         \begin{smallmatrix}
                           f_1 \\
                           g_1 \\
                         \end{smallmatrix}
                       \right)}\ar[d] & X_2\oplus Y_2\ar[r]^{(f_2, g_2)}\ar[d]^{(0,1)} &  X_3\ar[r]^{f_3} & \cdots \ar[r]^{f_{n-1}} & X_n\ar[r]^{f_n} & \Sigma X_1 \ar[d]\\
0\ar[r] & Y_2\ar[r]^{1} & Y_2\ar[r] & \cdots \ar[r] & 0 \ar[r] & 0\\
}$$ can be completed to a morphism of $n$-$\Sigma$-sequences by (N1)(b), (N2) and (N3). Thus there exists a morphism $g_2':X_3\rightarrow Y_2$ such that $g_2'f_2=0$ and $g_2'g_2=1$. Therefore, $g_2$ is a section.

Now we can assume that 
$$X_\bullet=(X_1\xrightarrow{\left(
                         \begin{smallmatrix}
                           f_1 \\
                           0 \\
                         \end{smallmatrix}
                       \right)}X_2\oplus Y_2\xrightarrow{\left(
                                                           \begin{smallmatrix}
                                                             f_{21} & 0 \\
                                                             f_{22} & g_{22} \\
                                                           \end{smallmatrix}
                                                         \right)
                       } X'_3\oplus g_2(Y_2)\xrightarrow{(f_{31},f_{32})} X_4\xrightarrow{f_4}\cdots\xrightarrow{f_{n-1}} X_n\xrightarrow{f_n}\Sigma X_1)$$
where $g_{22}$ is an isomorphism and $f_{32}=0$. Since $(f_{22},0)\left(
                         \begin{smallmatrix}
                           f_1 \\
                           0 \\
                         \end{smallmatrix}
                       \right)=f_{22}f_1=0$, there exists a morphism $(a,b): X'_3\oplus g_2(Y_2)\rightarrow g_2(Y_2)$, such that $(f_{22},0)=(a,b)\left(
                                                           \begin{smallmatrix}
                                                             f_{21} & 0 \\
                                                             f_{22} & g_{22} \\
                                                           \end{smallmatrix}
                                                         \right)$.
Thus $b=0$ and $f_{22}=af_{21}$. The following commutative diagram
$$\xymatrixcolsep{4pc}\xymatrixrowsep{3pc}\xymatrix{X_2\oplus Y_2\ar[r]^{\left(
                                                           \begin{smallmatrix}
                                                           f_{21} & 0 \\
                                                              0 & 1 \\
                                                           \end{smallmatrix}
                                                         \right)
                       }\ar@{=}[d] & X'_3\oplus Y_2\ar[r]^{(f_{31},0)}\ar[d]^{\left(
                                                                                          \begin{smallmatrix}
                                                                                            1 & 0 \\
                                                                                             a & g_{22} \\
                                                                                          \end{smallmatrix}
                                                                                        \right)
                       } & X_4\ar@{=}[d]\\
X_2\oplus Y_2\ar[r]^{\left(
                                                           \begin{smallmatrix}
                                                             f_{21} & 0 \\
                                                            f_{22} & g_{22} \\
                                                           \end{smallmatrix}
                                                         \right)
                       }& X'_3\oplus g_2(Y_2)\ar[r]^{(f_{31},0)} & X_4\\}$$
implies that $ X'_\bullet\oplus Y'_\bullet\cong X_\bullet$. By (N1)(a), $\Theta$ is closed under direct summands, thus we have $X'_\bullet\in\Theta$.
\end{proof}

\begin{lem}\label{1.2}



Let  $$\xymatrix{
X_\bullet\ar[d]^{\varphi_\bullet} & X_1 \ar[r]^{f_1}\ar[d]^{\varphi_1} & X_2 \ar[r]^{f_2}\ar[d]^{\varphi_2} & X_3 \ar[r]^{f_3}\ar[d]^{\varphi_3} & \cdots \ar[r]^{f_{n-1}}& X_n \ar[r]^{f_n}\ar[d]^{\varphi_n} & \Sigma X_1 \ar[d]^{\Sigma \varphi_1}\\
Y_\bullet\ar[d]^{\psi_\bullet} & Y_1 \ar[r]^{g_1}\ar[d]^{\psi_1} & Y_2 \ar[r]^{g_2}\ar[d]^{\psi_2} & Y_3 \ar[r]^{g_3}\ar[d]^{\psi_3} & \cdots \ar[r]^{g_{n-1}} & Y_n \ar[r]^{g_n}\ar[d]^{\psi_n}& \Sigma Y_1\ar[d]^{\Sigma\psi_1}\\
Z_\bullet & Z_1 \ar[r]^{h_1} & Z_2 \ar[r]^{h_2} & Z_3 \ar[r]^{h_3} & \cdots \ar[r]^{h_{n-1}} & Z_n \ar[r]^{h_n}& \Sigma Z_1\\
}$$ be a commutative diagram where each row is an $n$-angle. Assume that the morphism $\varphi_\bullet$ is a weak isomorphism. Then the mapping cone $C(\psi_\bullet\varphi_\bullet)$ is weakly isomorphic to the mapping cone $C(\psi_\bullet)$. Thus, $C(\psi_\bullet\varphi_\bullet)$ is an $n$-angle if and only if  so is $C(\psi_\bullet)$. In particular, if $\varphi_\bullet$ is an isomorphism, then $C(\psi_\bullet\varphi_\bullet)$ is isomorphic to $C(\psi_\bullet)$.
\end{lem}

\begin{proof}
It is easy to see that we have the following commutative diagram
$$\xymatrixcolsep{3.3pc}\xymatrixrowsep{3pc}\xymatrix{
 X_2\oplus Z_1\ar[r]^{\left(
                                         \begin{smallmatrix}
                                           -f_2 & 0 \\
                                           \psi_2\varphi_2 & h_1 \\
                                         \end{smallmatrix}
                                       \right)} \ar[d]^{\left(
                                         \begin{smallmatrix}
                                           \varphi_2 & 0 \\
                                           0 & 1 \\
                                         \end{smallmatrix}
                                       \right)}
& X_3\oplus Z_2\ar[r]^{\left(
                                         \begin{smallmatrix}
                                           -f_3 & 0 \\
                                           \psi_3\varphi_3 & h_2 \\
                                         \end{smallmatrix}
                                       \right)} \ar[d]^{\left(
                                         \begin{smallmatrix}
                                           \varphi_3 & 0 \\
                                           0 & 1 \\
                                         \end{smallmatrix}
                                       \right)}
& \cdots \ar[r]^{\left(
                                         \begin{smallmatrix}
                                           -f_n & 0 \\
                                           \psi_n\varphi_n & h_{n-1} \\
                                         \end{smallmatrix}
                                       \right)} & \Sigma X_1\oplus Z_n\ar[r]^{\left(
                                         \begin{smallmatrix}
                                           -\Sigma f_1 & 0 \\
                                           \Sigma\psi_1\Sigma\varphi_1 & h_n \\
                                         \end{smallmatrix}
                                       \right)} \ar[d]^{\left(
                                         \begin{smallmatrix}
                                           \Sigma\varphi_1 & 0 \\
                                           0 & 1 \\
                                         \end{smallmatrix}
                                       \right)}
& \Sigma X_2\oplus \Sigma Z_1 \ar[d]^{\left(
                                         \begin{smallmatrix}
                                           \Sigma\varphi_2 & 0 \\
                                           0 & 1 \\
                                         \end{smallmatrix}
                                       \right)} \\
 Y_2\oplus Z_1 \ar[r]^{\left(
                                         \begin{smallmatrix}
                                           -g_2 & 0 \\
                                           \psi_2 & h_1 \\
                                         \end{smallmatrix}
                                       \right)}
& Y_3\oplus Z_2 \ar[r]^{\left(
                                         \begin{smallmatrix}
                                           -g_3 & 0 \\
                                           \psi_3 & h_2 \\
                                         \end{smallmatrix}
                                       \right)}
& \cdots \ar[r]^{\left(
                                         \begin{smallmatrix}
                                           -g_n & 0 \\
                                           \psi_n & h_{n-1} \\
                                         \end{smallmatrix}
                                       \right)} & \Sigma Y_1\oplus Z_n \ar[r]^{\left(
                                         \begin{smallmatrix}
                                           -\Sigma g_1 & 0 \\
                                           \Sigma\psi_1 & h_n \\
                                         \end{smallmatrix}
                                       \right)}
&\Sigma Y_2\oplus \Sigma Z_1 \\
}$$ which implies that $C(\psi_\bullet\varphi_\bullet)$ is weakly isomorphic to $C(\psi_\bullet)$ since $\varphi_\bullet$ is a weak isomorphism. Following \cite[Proposition 2.5]{[GKO]},  the $n$-angles $X_\bullet, Y_\bullet$ and $Z_\bullet$ are exact, thus both $C(\psi_\bullet\varphi_\bullet)$ and $C(\psi_\bullet)$ are exact. By \cite[Lemma 2.4]{[GKO]} we infer that $C(\psi_\bullet\varphi_\bullet)$ is an $n$-angle if and only if so is $C(\psi_\bullet)$. The last assertion follows from the above diagram immediately.
\end{proof}

The following Lemma is a higher version of \cite[Lemma 1.16]{[BS]}.

\begin{lem}\label{1.3}
Let $$\widetilde{X}_\bullet=(X_1\oplus Y_1\xrightarrow{\left(
                              \begin{smallmatrix}
                                f_1 & 0 \\
                                \varphi_1 & g_1 \\
                              \end{smallmatrix}
                            \right)}
 X_2\oplus Y_2 \xrightarrow{\left(
                              \begin{smallmatrix}
                                f_2 & 0 \\
                                \varphi_2 & g_2 \\
                              \end{smallmatrix}
                            \right)}
 \cdots \xrightarrow{\left(
                            \begin{smallmatrix}
                               f_{n-1} & 0 \\
                                \varphi_{n-1} & g_{n-1} \\
                             \end{smallmatrix}
                           \right)}
 X_n\oplus Y_n \xrightarrow{\left(
                              \begin{smallmatrix}
                                 f_n & 0 \\
                                \varphi_n & g_n \\
                              \end{smallmatrix}
                            \right)}
 \Sigma X_1\oplus \Sigma Y_1 \\)
$$ be an $n$-angle. If $X_\bullet=(X_1\xrightarrow{f_1}X_2\xrightarrow{f_2}\cdots\xrightarrow{f_{n-1}}X_n\xrightarrow{f_n}\Sigma X_1)$ or
$Y_\bullet=(Y_1\xrightarrow{g_1}Y_2\xrightarrow{g_2}\cdots\xrightarrow{g_{n-1}}Y_n\xrightarrow{g_n}\Sigma Y_1)$ is a contractible $n$-angle, then $\widetilde{X}_\bullet$ is isomorphic to the direct sum of $X_\bullet$ and $Y_\bullet$.
\end{lem}

\begin{proof}
Assume that $X_\bullet$ is a contractible $n$-angle. Then by definition there exist morphisms $h_i:X_{i+1}\rightarrow X_i$ for $1\leq i\leq n-1$ and $h_n: \Sigma X_1\rightarrow X_n$ such that $1_{ X_1}= h_1 f_1+\Sigma^{-1}(f_nh_n)$ and $1_{X_i}=h_if_i+f_{i-1}h_{i-1}$ for $2\leq i\leq n$. Since $\varphi_{i+1}f_i=-g_{i+1}\varphi_i$ for $1\leq i\leq n-1$ and $\Sigma\varphi_1 \cdot f_n=-\Sigma g_1\cdot\varphi_n$, it is not hard to check that the following diagram
$$\xymatrixcolsep{3.0pc}\xymatrixrowsep{3.5pc}\xymatrix{
X_1\oplus Y_1\ar[r]^{\left(
                              \begin{smallmatrix}
                                f_1 & 0 \\
                                \varphi_1 & g_1 \\
                              \end{smallmatrix}
                            \right)}\ar[d]^{\left(
                              \begin{smallmatrix}
                                1 & 0 \\
                                -\Sigma^{-1}(\varphi_nh_n) & 1 \\
                              \end{smallmatrix}
                            \right)} &
 X_2\oplus Y_2 \ar[r]^{\left(
                              \begin{smallmatrix}
                                f_2 & 0 \\
                                \varphi_2 & g_2 \\
                              \end{smallmatrix}
                            \right)} \ar[d]^{\left(
                              \begin{smallmatrix}
                                1 & 0 \\
                                -\varphi_1h_1 & 1 \\
                              \end{smallmatrix}
                            \right)} &
 \cdots \ar[r]^{\left(
                            \begin{smallmatrix}
                               f_{n-1} & 0 \\
                                \varphi_{n-1} & g_{n-1} \\
                             \end{smallmatrix}
                           \right)}  &
  X_n\oplus Y_n \ar[r]^{\left(
                              \begin{smallmatrix}
                                 f_n & 0 \\
                                \varphi_n & g_n \\
                              \end{smallmatrix}
                            \right)} \ar[d]^{\left(
                              \begin{smallmatrix}
                                1 & 0 \\
                               -\varphi_{n-1}h_{n-1} & 1 \\
                              \end{smallmatrix}
                            \right)} &
 \Sigma X_1\oplus \Sigma Y_1 \ar[d]^{\left(
                              \begin{smallmatrix}
                                1 & 0 \\
                                -\varphi_nh_n & 1 \\
                              \end{smallmatrix}
                            \right)}\\
 X_1\oplus Y_1\ar[r]^{\left(
                              \begin{smallmatrix}
                                f_1 & 0 \\
                                0 & g_1 \\
                              \end{smallmatrix}
                            \right)} &
 X_2\oplus Y_2 \ar[r]^{\left(
                              \begin{smallmatrix}
                                f_2 & 0 \\
                                0 & g_2 \\
                              \end{smallmatrix}
                            \right)}  &
 \cdots \ar[r]^{\left(
                            \begin{smallmatrix}
                               f_{n-1} & 0 \\
                                0 & g_{n-1} \\
                             \end{smallmatrix}
                           \right)}  &
 X_n\oplus Y_n \ar[r]^{\left(
                              \begin{smallmatrix}
                                 f_n & 0 \\
                                0 & g_n \\
                              \end{smallmatrix}
                            \right)} &
 \Sigma X_1\oplus \Sigma Y_1 \\
 }$$is commutative, which implies that $\widetilde{X}_\bullet\cong X_\bullet\oplus Y_\bullet$. If $Y_\bullet$ is contractible, we can prove it similarly.
\end{proof}

\begin{lem}\label{1.4}
Let $\varphi_\bullet=\left(
                      \begin{array}{cc}
                       \alpha_\bullet & \beta_\bullet \\
                        \gamma_\bullet & \delta_\bullet\\
                      \end{array}
                    \right):X_\bullet\oplus X'_\bullet\rightarrow Y_\bullet\oplus Y'_\bullet
$ and $\alpha_\bullet:X_\bullet\rightarrow Y_\bullet$  be morphisms of $n$-angles, where $X'_\bullet$ and $Y'_\bullet$ are contractible, then the mapping cone
$$C(\varphi_\bullet)\cong C(\alpha_\bullet)\oplus X'_\bullet[1]\oplus Y'_\bullet.$$
\end{lem}

\begin{proof}
Assume that $\varphi'_\bullet=\left(
                      \begin{array}{cc}
                       \beta_\bullet & \alpha_\bullet \\
                       \delta_\bullet & \gamma_\bullet\\
                      \end{array}
                    \right):X'_\bullet\oplus X_\bullet\rightarrow Y_\bullet\oplus Y'_\bullet$, then $\varphi_\bullet=\varphi'_\bullet\left(
                                                                                                                                       \begin{array}{cc}
                                                                                                                                         0 & 1 \\
                                                                                                                                         1 & 0 \\
                                                                                                                                       \end{array}
                                                                                                                                     \right)
                    $. Thus we have $$C(\varphi_\bullet)\cong C(\varphi'_\bullet)\cong X'_\bullet[1]\oplus C(\alpha_\bullet)\oplus Y'_\bullet,$$
where the first isomorphism follows from Lemma \ref{1.2} and the second isomorphism uses Lemma \ref{1.3} twice.
\end{proof}

\section{Main theorem}

In this section, we first recall the construction of idempotent completion of an additive category and some related facts from \cite{[BS],[B]}, then state and prove our main theorem.

An additive category $\mathcal{C}$ is called {\em idempotent complete} if for each object $A$ in $\mathcal{C}$ and for each idempotent $e:A\rightarrow A$, we have $A=\mbox{Im}(e)\oplus \mbox{Ker}(e)$.

Let $\mathcal{C}$ be an additive category. The {\em idempotent completion} of $\mathcal{C}$ is the category $\widetilde{\mathcal{C}}$ defined as follows. Objects of $\mathcal{\widetilde{C}}$ are pairs $(A,e)$, where $A$ is an object in $\mathcal{C}$ and $e: A\rightarrow A$ is an idempotent. A morphism in $\widetilde{\mathcal{C}}$ from $(A,e)$ to $(B,f)$ is in the form of $fpe: A\rightarrow B$ for some morphism $p:A\rightarrow B$ in $\mathcal{C}$.

Assume that $\mathcal{D}$ is another additive category. An additive functor $F:\mathcal{C}\rightarrow \mathcal{D}$ yields an additive  functor $\widetilde{F}:\mathcal{\widetilde{C}}\rightarrow \mathcal{\widetilde{D}}$, by setting $\widetilde{F}(A,e)=(FA,Fe)$ and $\widetilde{F}(fp e)=F(f)F(p)F(e)$. Suppose $G:\mathcal{D}\rightarrow\mathcal{E}$ is another additive functor, then it is clear that $\widetilde{GF}=\widetilde{G}\widetilde{F}$.

Given two additive functors $F,H:\mathcal{C}\rightarrow\mathcal{D}$, a natural transformation $\alpha: F\rightarrow H$ yields a unique natural transformation $\widetilde{\alpha}:\widetilde{F}\rightarrow\widetilde{H}$, by putting $\widetilde{\alpha}_{(A,e)}=H(e)\alpha_AF(e)$.

 The assignment $A\mapsto (A,1)$ defines an additive functor $\iota: \mathcal{\mathcal{C}}\rightarrow\widetilde{\mathcal{C}}$. It is easy to see that $\iota F=\widetilde{F}\iota$, where $F:\mathcal{C}\rightarrow \mathcal{D}$ is an additive functor. The following is well known.

 \begin{prop}\label{2.1}
 The category $\widetilde{\mathcal{C}}$ is an idempotent complete  additive category and the functor $\iota: \mathcal{\mathcal{C}}\rightarrow\widetilde{\mathcal{C}}$ is fully faithful. In particular, $\mathcal{C}$ is idempotent complete if and only if $\iota:\mathcal{C}\rightarrow \widetilde{\mathcal{C}}$ is an equivalence.  Moreover, given an idempotent complete additive category $\mathcal{D}$ and an additive functor $F:\mathcal{C}\rightarrow\mathcal{D}$, there exists a unique additive functor $G:\widetilde{\mathcal{C}}\rightarrow \mathcal{D}$ such that $F=G\iota$.
 \end{prop}

\begin{rem}
Since the functor $\iota: \mathcal{\mathcal{C}}\rightarrow\widetilde{\mathcal{C}}$ is fully faithful, for convenience we view $\mathcal{C}$ as a full subcategory of $\widetilde{\mathcal{C}}$. Thus for each object $X\in\widetilde{\mathcal{C}}$, there exists an object $X'\in\widetilde{\mathcal{C}}$ such that $X\oplus X'\in\mathcal{C}$. In fact, if $X=(A,e)$, then we can take $X'=(A,1-e)$ and $X\oplus X'\cong A\in\mathcal{C}$.
\end{rem}


Let $(\mathcal{C},\Sigma,\Theta)$ and $(\mathcal{C}',\Sigma',\Theta')$ be two $n$-angulated categories. An additive functor $F:\mathcal{C}\rightarrow \mathcal{C}'$ is called $n$-$angulated$ (see \cite{[BT2]}) if there exists a natural isomorphism $\alpha:F\Sigma\rightarrow\Sigma'F$ and $F$ preserves $n$-angles: if
$X_1\xrightarrow{f_1}X_2\xrightarrow{f_2}\cdots\xrightarrow{f_{n-1}}X_n\xrightarrow{f_n}\Sigma X_1$
is an $n$-angle in $\Theta$, then
$FX_1\xrightarrow{Ff_1}FX_2\xrightarrow{Ff_2}\cdots\xrightarrow{Ff_{n-1}}FX_n\xrightarrow{\alpha_{X_1}Ff_n}\Sigma' FX_1$
is an $n$-angle in $\Theta'$.

We are now in a position to state and prove our main result.

\begin{thm}\label{thm}
Let $(\mathcal{C},\Sigma,\Theta)$ be an $n$-angulated category. Then its idempotent completion $\widetilde{\mathcal{C}}$ admits a unique $n$-angulated structure such that the natural functor $\iota:\mathcal{C}\rightarrow \widetilde{\mathcal{C}}$ is $n$-angulated. Moreover, given an idempotent complete $n$-angulated category $\mathcal{D}$ and an $n$-angulated functor $F:\mathcal{C}\rightarrow\mathcal{D}$, there exists a unique $n$-angulated functor $G:\widetilde{\mathcal{C}}\rightarrow \mathcal{D}$ such that $F=G\iota$.
\end{thm}

\begin{proof} The isomorphism $\Sigma$ of $\mathcal{C}$ induces an isomorphism $\widetilde{\Sigma}$ of $\widetilde{\mathcal{C}}$, by setting $\widetilde{\Sigma}(A,e)=(\Sigma A,\Sigma e)$. Denote by $\widetilde{\Theta}$ the idempotent completion of $\Theta$. Here we only care about the objects of $\widetilde{\Theta}$. We first show that $(\widetilde{\mathcal{C}},\widetilde{\Sigma},\widetilde{\Theta})$ is an $n$-angulated category.

(N1)(a) It is clear that $\widetilde{\Theta}$ is closed under isomorphisms, direct sums and direct summands.

(N1)(b) Let $X$ be an object in $\widetilde{\mathcal{C}}$. There exists an object $X'\in\widetilde{\mathcal{C}}$ such that $X\oplus X'\in \mathcal{C}$.
The trivial $n$-angle $X\oplus X'\xrightarrow{1}X\oplus X'\rightarrow 0\rightarrow\cdots\rightarrow 0\rightarrow\Sigma (X\oplus  X')\in\Theta$ implies that
$X\xrightarrow{1}X\rightarrow 0\rightarrow\cdots\rightarrow 0\rightarrow\widetilde{\Sigma} X\in\widetilde{\Theta}$.

(N2) It follows from the definition.

(N3) Given a commutative diagram
$$\begin{gathered}\xymatrix{
X_1 \ar[r]^{f_1}\ar[d]^{\varphi_1} & X_2 \ar[r]^{f_2}\ar[d]^{\varphi_2} & X_3 \ar[r]^{f_3} & \cdots \ar[r]^{f_{n-1}}& X_n \ar[r]^{f_n} & \widetilde{\Sigma} X_1 \ar[d]^{\widetilde{\Sigma} \varphi_1}\\
Y_1 \ar[r]^{g_1} & Y_2 \ar[r]^{g_2} & Y_3 \ar[r]^{g_3} & \cdots \ar[r]^{g_{n-1}} & Y_n \ar[r]^{g_n}& \widetilde{\Sigma} Y_1\\
}\end{gathered}\eqno{(3.1)}$$ with rows in $\widetilde{\Theta}$. There exist two objects $X'_\bullet$ and $Y'_\bullet$ in $\widetilde{\Theta}$ such that $X_\bullet\oplus X_\bullet', Y_\bullet\oplus Y_\bullet'\in\Theta$. Thus we have four canonical morphisms of $n$-$\widetilde{\Sigma}$-sequences $i_\bullet: X_\bullet\rightarrow X_\bullet\oplus X_\bullet'$, $j_\bullet:Y_\bullet\rightarrow Y_\bullet\oplus Y'_\bullet$, $p_\bullet:X_\bullet\oplus X_\bullet'\rightarrow X_\bullet$ and $q_\bullet:Y_\bullet\oplus Y'_\bullet\rightarrow Y_\bullet$ such that $p_\bullet i_\bullet=1$ and $q_\bullet j_\bullet=1$. By (N3), we can complete the pair $(j_1\varphi_1p_1,j_2\varphi_2p_2)$ to a morphism $\widetilde{\varphi}_\bullet:X_\bullet\oplus X_\bullet'\rightarrow Y_\bullet\oplus Y_\bullet'$. It is easy to see that $q_\bullet\widetilde{\varphi}_\bullet i_\bullet: X_\bullet\rightarrow Y_\bullet$ is a morphism of $n$-$\widetilde{\Sigma}$-sequences extending $(\varphi_1,\varphi_2)$.

(N1)(c) For each morphism $f_1:X_1\rightarrow X_2$ in $\widetilde{\mathcal{C}}$, we choose two objects $X_1',X_2'\in\widetilde{\mathcal{C}}$ such that $\left(
                                                                                                                                                  \begin{smallmatrix}
                                                                                                                                                    f_1 & 0 \\
                                                                                                                                                    0 & 0 \\
                                                                                                                                                  \end{smallmatrix}
                                                                                                                                                \right):X_1\oplus
X_1'\rightarrow X_2\oplus X_2'$ is a morphism in $\mathcal{C}$. By (N1)(c), we assume that
$$\widetilde{X}_\bullet=(X_1\oplus X_1'\xrightarrow{\left(
 \begin{smallmatrix}
 f_1 & 0 \\
  0 & 0 \\
  \end{smallmatrix}
  \right)}X_2\oplus X_2'\xrightarrow{f_2} X_3\xrightarrow{f_3}\cdots\xrightarrow{f_{n-1}}X_n\xrightarrow{f_n} \Sigma (X_1\oplus X_1'))$$
is an $n$-angle in $\Theta$. Then $\widetilde{X}_\bullet\in\widetilde{\Theta}$. Lemma \ref{1.1} implies that $\widetilde{X}_\bullet\cong X_\bullet\oplus X'_\bullet$, where
 $$X_\bullet=(X_1\xrightarrow{f_1}X_2\xrightarrow{f_2'} X_3'\xrightarrow{f_3'} X_4\xrightarrow{f_4}\cdots\xrightarrow{f_{n-2}} X_{n-1}\xrightarrow{f_{n-1}'} X_n'\xrightarrow{f_n'}\widetilde{\Sigma} X_1),$$
 $$X'_\bullet=(X_1'\xrightarrow{0} X_2'\xrightarrow{1} X_2'\rightarrow 0\rightarrow\cdots\rightarrow 0\rightarrow \widetilde{\Sigma} X'_1\xrightarrow{1}\widetilde{\Sigma} X'_1).$$
So $X_\bullet$ belongs to $\widetilde{\Theta}$.

Therefore, $(\widetilde{\mathcal{C}},\widetilde{\Sigma},\widetilde{\Theta})$ is a pre-$n$-angulated category.

(N4) Given a commutative diagram (3.1) with rows in $\widetilde{\Theta}$.
For $i=1,2$, we choose $X_i',Y_i'\in \widetilde{\mathcal{C}}$ such that $\left(
                                                                                                                                                  \begin{smallmatrix}
                                                                                                                                                    f_1 & 0 \\
                                                                                                                                                    0 & 0 \\
                                                                                                                                                  \end{smallmatrix}
                                                                                                                                                \right):X_1\oplus
X_1'\rightarrow X_2\oplus X_2'$ and $\left(
                                                                                                                                                  \begin{smallmatrix}
                                                                                                                                                    g_1 & 0 \\
                                                                                                                                                    0 & 0 \\
                                                                                                                                                  \end{smallmatrix}
                                                                                                                                                \right):Y_1\oplus
Y_1'\rightarrow Y_2\oplus Y_2'$ are morphisms in $\mathcal{C}$.  By axiom (N1)(c), we can assume that $\widetilde{X}_\bullet$ and $\widetilde{Y}_\bullet$ are $n$-angles in $\Theta$ such that the first morphisms are $\left(
                                                                                                                                                  \begin{smallmatrix}
                                                                                                                                                    f_1 & 0 \\
                                                                                                                                                    0 & 0 \\
                                                                                                                                                  \end{smallmatrix}
                                                                                                                                                \right)$
and $\left(
                                                                                                                                                  \begin{smallmatrix}
                                                                                                                                                    g_1 & 0 \\
                                                                                                                                                    0 & 0 \\
                                                                                                                                                  \end{smallmatrix}
                                                                                                                                                \right)$
respectively.  By (N4), the pair $(\left(
                      \begin{smallmatrix}
                       \varphi_1 & 0 \\
                       0 & 0 \\
                      \end{smallmatrix}
                    \right), \left(
                      \begin{smallmatrix}
                       \varphi_2 & 0 \\
                       0 & 0 \\
                      \end{smallmatrix}
                    \right))$ can be completed to a morphism 
 $\widetilde{\varphi}_\bullet: \widetilde{X}_\bullet\rightarrow \widetilde{Y}_\bullet$ such that the mapping cone $ C(\widetilde{\varphi}_\bullet)\in\Theta$.
Assume that $$X'_\bullet=(X_1'\xrightarrow{0} X_2'\xrightarrow{1} X_2'\rightarrow 0\rightarrow\cdots\rightarrow 0\rightarrow \widetilde{\Sigma} X'_1\xrightarrow{1}\widetilde{\Sigma} X'_1),$$
$$Y'_\bullet=(Y_1'\xrightarrow{0} Y_2'\xrightarrow{1} Y_2'\rightarrow 0\rightarrow\cdots\rightarrow 0\rightarrow \widetilde{\Sigma} Y'_1\xrightarrow{1}\widetilde{\Sigma} Y'_1).$$
 Then both $X'_\bullet$ and $Y'_\bullet$ are contractible $n$-angles in $\widetilde{\Theta}$. By (N3) we have two weak isomorphisms $\psi_\bullet: X_\bullet\oplus X_\bullet'\rightarrow\widetilde{X}_\bullet$ and $\theta_\bullet: \widetilde{Y}_\bullet\rightarrow Y_\bullet\oplus Y'_\bullet$. Since $C(\widetilde{\varphi}_\bullet)\in\Theta\subseteq\widetilde{\Theta}$, it follows from Lemma \ref{1.2} and its dual that the mapping cone $C(\theta_\bullet\widetilde{\varphi}_\bullet\psi_\bullet)\in\widetilde{\Theta}$.
Suppose that $\theta_\bullet\widetilde{\varphi}_\bullet\psi_\bullet=\left(
                                             \begin{array}{cc}
                                               \alpha_\bullet & \beta_\bullet \\
                                               \gamma_\bullet & \delta_\bullet \\
                                             \end{array}
                                           \right):X_\bullet\oplus X_\bullet'\rightarrow Y_\bullet\oplus Y_\bullet'
$, then it is easy to check that $\alpha_\bullet: X_\bullet\rightarrow Y_\bullet$ is a morphism of $n$-$\widetilde{\Sigma}$-sequences where $\alpha_i=\varphi_i$ for $i=1,2$. Since $C(\theta_\bullet\widetilde{\varphi}_\bullet\psi_\bullet)\cong C(\alpha_\bullet)\oplus X_\bullet'[1]\oplus Y_\bullet'$ by Lemma \ref{1.4},  we obtain that $C(\alpha_\bullet)\in\widetilde{\Theta}$. So (N4) holds. Consequently, $(\widetilde{\mathcal{C}},\widetilde{\Sigma},\widetilde{\Theta})$ is an $n$-angulated category.

It is clear that the functor $\iota:\mathcal{C}\rightarrow\widetilde{\mathcal{C}}$ is $n$-angulated. Suppose that $\widetilde{\Theta}'$ is another $n$-angulation of $(\widetilde{\mathcal{C}},\widetilde{\Sigma})$ such that the functor $\iota$ is $n$-angulated. Since $\iota$ is $n$-angulated and $\widetilde{\Theta}'$ is closed under direct summands, we infer that $\widetilde{\Theta}\subseteq\widetilde{\Theta}'$. It follows from \cite[Propositon 2.5 (c)]{[GKO]} that $\widetilde{\Theta}=\widetilde{\Theta}'$. This proves the uniqueness of the $n$-angulated structure.

Given an idempotent complete $n$-angulated category $(\mathcal{D},\Sigma',\Phi)$ and an $n$-angulated functor $F:\mathcal{C}\rightarrow\mathcal{D}$, by Proposition \ref{2.1} there exists a unique additive functor $G:\widetilde{\mathcal{C}}\rightarrow \mathcal{D}$ such that $F=G\iota$. We only need to show that the functor $G$ is $n$-angulated. 

Assume that $\alpha:F\Sigma\rightarrow\Sigma' F$ is the natural isomorphism given by the $n$-angulated functor $F$.
We define $\widetilde{\alpha}_{(A,e)}=\Sigma'F(e)\cdot\alpha_{A}\cdot F\Sigma(e)$ for each $(A,e)\in\widetilde{C}$, then  $\widetilde{\alpha}: G\widetilde{\Sigma}\rightarrow\Sigma' G$ is a natural isomorphism. 
 Suppose that
$$X_\bullet=(X_1\xrightarrow{f_1} X_2\xrightarrow{f_2} \cdots\xrightarrow{f_{n-1}} X_n\xrightarrow{f_n} \widetilde{\Sigma} X_1)$$
 is an $n$-angle in $\widetilde{\Theta}$, there exists
another $n$-angle $$X'_\bullet=(X'_1\xrightarrow{f'_1} X'_2\xrightarrow{f'_2} \cdots\xrightarrow{f'_{n-1}} X'_n\xrightarrow{f'_n} \widetilde{\Sigma} X'_1)$$
in $\widetilde{\Theta}$ such that $X_\bullet\oplus X'_\bullet\in\Theta$.  Since $F=G\iota$ is $n$-angulated,  the following sequence $$F(X_1\oplus X_1')\xrightarrow{F\left(
                                                                                     \begin{smallmatrix}
                                                                                       f_1 & 0 \\
                                                                                       0 & f_1' \\
                                                                                     \end{smallmatrix}
                                                                                   \right)} F(X_2\oplus X_2')\xrightarrow{F\left(
                                                                                     \begin{smallmatrix}
                                                                                       f_2 & 0 \\
                                                                                       0 & f_2' \\
                                                                                     \end{smallmatrix}
                                                                                   \right)} \cdots$$
$$\cdots\xrightarrow{F\left(
                                                                                     \begin{smallmatrix}
                                                                                       f_{n-1} & 0 \\
                                                                                       0 & f_{n-1}' \\
                                                                                     \end{smallmatrix}
                                                                                   \right)} F(X_n\oplus X_n')\xrightarrow{\alpha_{X_1\oplus X_1'}F\left(
                                                                                     \begin{smallmatrix}
                                                                                       f_n & 0 \\
                                                                                       0 & f_n' \\
                                                                                     \end{smallmatrix}
                                                                                   \right)} \Sigma' F(X_1\oplus X_1')$$ is an $n$-angle in $\Phi$.
Because $\mathcal{D}$ is idempotent complete,
we can easily infer that
$$GX_1\xrightarrow{Gf_1} GX_2\xrightarrow{Gf_2} \cdots\xrightarrow{Gf_{n-1}} GX_n\xrightarrow{\widetilde{\alpha}_{X_1}Gf_n} \Sigma' GX_1$$ is an $n$-angle in $\Phi$.
This completes the whole proof.
\end{proof}

\vspace{2mm}\noindent {\bf Acknowledgements}  The work was done when the author visited University of Stuttgart. The author thanks Steffen Koenig for warm hospitality during his stay in Stuttgart and for helpful discussions and remarks, especially in  writing  the  introduction.

\end{document}